\documentclass[11pt, reqno]{amsart}
\usepackage{amscd}
\usepackage{amsmath}
\usepackage{amssymb}
\usepackage{latexsym}
\usepackage{amsthm}
\usepackage{graphicx}

\usepackage[all]{xy}       

    \SelectTips{cm}{10}     

    \everyxy={<2.5em,0em>:} 

    \xyoption{web}          

\usepackage{hyperref}
\usepackage{cleveref}
\allowdisplaybreaks

\newcommand{\arxiv}[1]{\href{http://arxiv.org/abs/#1}{\tt arXiv:\nolinkurl{#1}}}
\newcommand{\arXiv}[1]{\href{http://arxiv.org/abs/#1}{\tt arXiv:\nolinkurl{#1}}}

\newtheorem{theorem}{Theorem}[section]
\newtheorem{lemma}[theorem]{Lemma}

\newtheorem{definition}[theorem]{Definition}

\newtheorem{proposition}[theorem]{Proposition}
\newtheorem{corollary}[theorem]{Corollary}

\newtheorem{question}[theorem]{Question}

\theoremstyle{remark}
\newtheorem{remark}[theorem]{Remark}

\numberwithin{equation}{section}

\newcommand{\nc}{\newcommand}

\nc{\flags}{\mathcal{F}}

\def\wt{\text{wt}}

\nc{\KP}{\operatorname{KP}}

\newcommand{\overbar}[1]{\mkern 1.5mu\overline{\mkern-1.5mu#1\mkern-1.5mu}\mkern 1.5mu}
\def\ii{{\bf i}}
\nc{\re}{re}
\def\j{{\bf j}}

\def\N{\mathbb{N}}
\def\Q{\mathbb{Q}}

\def\C{\mathcal{C}}

\def\Z{\mathbb{Z}}

\def\k{k}

\def\t{\tilde}

\def\A{\mathcal{A}}

\def\sB{\mathcal{B}}

\def\uj{\bf j}

\def\P{\mathcal{P}} 
\def\L{\mathcal{L}}

\def\n{\mathfrak {n}}

\nc{\co}{\nabla}

\def\a{\alpha}
\def\b{\beta}
\def\e{\epsilon}

\def\la{\lambda}

\def\ga{\gamma}

\def\th{\theta}
\def\f{\mathbf{f}}
\def\kk{\mathbf{k}}
\def\g{\mathfrak{g}}
\def\B{\mathbf{B}}

\def\D{\mathbb{D}}

\def\Hom{\operatorname{Hom}}

\def\Ind{\operatorname{Ind}}

\def\seq{\,\mbox{Seq}\,}

\def\soc{\operatorname{soc}}
\def\End{\operatorname{End}}
\def\hd{\operatorname{hd}}
\def\Seq{\,\mbox{Seq}\,}

\def\Res{\operatorname{Res}}

\def\diag{\,\mbox{diag}\,}

\def\mods{\mbox{-mod}}

\def\id{\mbox{id}}

\newcommand{\map}[2]{\,{:}\,#1\!\longrightarrow\!#2}

\def\inv{^{-1}}

\numberwithin{equation}{section}

\title[folding KLR algebras]{Folding KLR algebras}
\address{}\email{maths@petermc.net}
\author{Peter J McNamara}
\date{\today}

\begin{document}

 \begin{abstract}
 This paper develops the theory of KLR algebras with a Dynkin diagram automorphism. This is foundational material intended to allow folding techniques in the theory of KLR algebras.
 \end{abstract}
%

\maketitle

\section{Introduction} 

In Lie theory, the process of folding by a Dynkin diagram automorphism is a technique that can be used to extend theorems originally proved for symmetric Cartan data to all Lie types. This paper develops the theory of folding for KLR algebras.

KLR algebras (named after Khovanov, Lauda and Rouquier) are a family of graded algebras introduced in \cite{khovanovlauda,rouquier} for the purposes of categorifying quantised enveloping algebras.
They also appear in the literature under the name of quiver Hecke algebras.

While KLR algebras exist for arbitrary symmetrisable Cartan data, 
it is known that the KLR algebras in symmetric types have a richer theory with more desirable properties. This is usually a consequence of the geometric interpretation of symmetric KLR algebras \cite{vv,rouquier2,maksimau} or through the theory of $R$-matrices of \cite{kkk}. 
Thus we believe that incorporating a diagram automorphism into the narrative and using the technique of folding is an important way to think about categorified quantum groups in nonsymmetric types, as an alternative to working with nonsymmetric KLR algebras.

The folding constructions performed in this paper are modelled on those of \cite[Chapter 2]{lusztigbook}, where Lusztig constructs the canonical basis using perverse sheaves and a diagram automorphism. In fact when one considers the geometric interpretation of KLR algebras as extension algebras and works over the field $\overbar\Q_l$, the category $\P_\nu$ which we study is equivalent to Lusztig's $\tilde{\mathcal{Q}}_V$.

Our main aim in this paper is to develop the theory of folded KLR algebras, to a depth comparable to that of \cite{khovanovlauda}.  Our main theorems are the categorification theorems, Theorems \ref{projiso} and \ref{simiso}. Our proof is different from \cite{khovanovlauda} in that we do not rely on the quantum Gabber-Kac theorem. Instead, we first identify an appropriate class of simple objects with the crystal $B(\infty)$, using the Kashiwara-Saito characterisation of $B(\infty)$. This identification is Theorem \ref{crystal}, and generalises \cite{laudavazirani}.

We also conclude with a section proving that this categorifcation provides us with a \emph{basis of canonical type}. This concept of a basis of canonical type is motivated from the definition given in \cite{baumann} and is a strengthening of the notion of a perfect basis.

An application of this work to the KLR categorification of cluster algebras will appear in a forthcoming paper \cite{cluster}.

\section{KLR Algebras}

A Cartan datum $(I,\cdot)$ is a set $I$ together with a symmetric bilinear form on the free abelian group $\Z I$, denoted $i\cdot j$ such that
\begin{enumerate}
 \item $i\cdot i \in \{2,4,6,\cdots\}$ for all $i\in I$
 \item $2 \frac{i\cdot j}{i\cdot i} \in \{0,-1,-2,\ldots\}$ for all distinct $i,j\in I$.
\end{enumerate}
Let $d_i=i\cdot i/2$ and $c_{ij}=i\cdot j/d_i$. Then the matrix $C=(c_{ij})_{i,j\in I}$ is a symmetrisable Cartan matrix with $D=\diag(d_i)_{i\in I}$ a symmetrising matrix.

Let $a$ be an automorphism of the Cartan datum $(I,\cdot)$ such that $i\cdot j=0$ if $i$ and $j$ are in the same $a$-orbit. 
Let $n$ be the order of $a$. We assume that $n$ is finite (this is of course automatic if $I$ is finite).



Out of the data of $C$ and $a$, we construct another Cartan datum. Let $J$ be the set of orbits in $I$. We can embed $\Z J$ inside $\Z I$ by sending $j\in J$ to $\sum_{i\in j} i$. If we restrict the symmetric form on $\Z I$ to $\Z J$ we place the structure of a Cartan datum on $J$.
It is known that any Cartan datum $(J,\cdot)$ arises from such a construction where the Cartan datum $(I,\cdot)$ satisfies $i\cdot i=2$ for all $i\in I$. Such Cartan data $(I,\cdot)$ are called symmetric.

Define, for any $\nu\in\N I$,
\[
 \Seq(\nu)=\{\ii=(\ii_1,\ldots,\ii_{|\nu|})\in I^{|\nu|}\mid\sum_{j=1}^{|\nu|} \ii_j=\nu\}.
\]
This is acted upon by the symmetric group $S_{|\nu|}$ in which the adjacent transposition $(i,i+1)$ is denoted $s_i$.

Let $k$ be an algebraically closed field whose characteristic does not divide $n$.
In a similar vein to how simple modules for a KLR algebra over any field are absolutely irreducible, we expect here that we only need to assume that $k$ contains $n$ $n$-th roots of unity. But we do not pursue this question in this paper.

To each $i,j\in I$, we define polynomials $Q_{ij}(u,v)\in k[u,v]$ such that for all $i,j\in I$,
\begin{enumerate}
 \item $Q_{ii}(u,v)=0$
\item If $u$ has degree $d_i$ and $v$ has degree $d_j$ then $Q_{ij}$ is a homogeneous polynomial of degree $-d_i c_{ij}=-d_j c_{ji}$ such that the coefficients of $u^{-c_{ij}}$ and $v^{-c_{ji}}$ are both nonzero.
\item $Q_{ij}(u,v)=Q_{ji}(v,u)$.
\item $Q_{a(i)a(j)}(u,v) = Q_{ij}(u,v)$.
\end{enumerate}

The KLR algebras are defined in terms of this family of polynomials $Q_{i,j}$, though this dependence is suppressed from the notation.
There is a diagrammatic approach to presenting the following generators and relations, which the reader may find more convenient. These diagrams can be found in \cite{khovanovlauda,kl2,tingleywebster}.

\begin{definition}
The KLR algebra $R(\nu)$ is the associative $\k$-algebra generated by elements $e_{\bf i}$, $y_j$, $\tau_k$ with ${\bf i}\in \seq(\nu)$, $1\leq j\leq |\nu|$ and $1\leq k< |\nu|$, subject to the relations
\begin{equation} \label{eq:KLR}
\begin{aligned}
& e_\ii e_{\uj} = \delta_{\ii, \uj} e_\ii, \ \
\sum_{\ii \in {\rm Seq}(\nu)}  e_\ii = 1, \\
& y_{k} y_{l} = y_{l} y_{k}, \ \ y_{k} e_\ii = e_\ii y_{k}, \\
& \tau_{l} e_\ii = e_{s_{l}\ii} \tau_{l}, \ \ \tau_{k} \tau_{l} =
\tau_{l} \tau_{k} \ \ \text{if} \ |k-l|>1, \\
& \tau_{k}^2 e_\ii = Q_{\ii_{k}, \ii_{k+1}} (y_{k}, y_{k+1})e_\ii, \\
& (\tau_{k} y_{l} - y_{s_k(l)} \tau_{k}) e_\ii = \begin{cases}
-e_\ii \ \ & \text{if} \ l=k, \ii_{k} = \ii_{k+1}, \\
e_\ii \ \ & \text{if} \ l=k+1, \ii_{k}=\ii_{k+1}, \\
0 \ \ & \text{otherwise},
\end{cases} \\[.5ex]
& (\tau_{k+1} \tau_{k} \tau_{k+1}-\tau_{k} \tau_{k+1} \tau_{k}) e_\ii\\
& =\begin{cases} \dfrac{Q_{\ii_{k}, \ii_{k+1}}(y_{k},
y_{k+1}) - Q_{\ii_{k}, \ii_{k+1}}(y_{k+2}, y_{k+1})}
{y_{k} - y_{k+2}}e_\ii\ \ & \text{if} \
\ii_{k} = \ii_{k+2}, \\
0 \ \ & \text{otherwise}.
\end{cases}
\end{aligned}
\end{equation}
\end{definition}

If $w$ is a permutation in $S_{|\nu|}$, write $w$ as a reduced product of simple reflections $w=s_{i_1}\ldots s_{i_n}$ and define $\tau_w=\tau_{i_1}\ldots\tau_{i_n}$. In general this depends on the choice of a reduced product but we will only use it in this paper for those $w$ for which $\tau_w$ is well-defined.

The KLR algebras $R(\nu)$ are $\Z$-graded, where $e_{\ii}$ is of degree zero, $y_j e_\ii$ is of degree $\ii_j \cdot \ii_j$ and $\tau_k e_\ii$ is of degree $-\ii_k \cdot \ii_{k+1}$. All $R(\nu)$-modules will always be assumed to be graded left modules. 
If $V$ is any $\Z$-graded vector space, we use $V_i$ to denote its $i$-th graded piece.
Given a $R(\nu)$-module $M$, its grading shift is denoted $qM$, this is the module with $(qM)_i=M_{i-1}$.

If $M$ and $N$ are two $R(\nu)$-modules and $d\in \Z$, we define $\Hom_{R(\nu)}(M,N)_d$ to be the space of graded $R(\nu)$-module homomorphisms from $M$ to $q^dN$.

Because of the condition $Q_{a(i)a(j)}(u,v) = Q_{ij}(u,v)$ on the polynomials $Q_{ij}$, the automorphism $a$ of the Cartan datum $(I,\cdot)$ induces an isomorphism $R(a\nu)\cong R(\nu)$. The most important case for us will be when $a\nu=\nu$, when $a$ induces an automorphism of the algebra $R(\nu)$, which we shall also call $a$.

%

Now consider $\nu$ such that $a\nu=\nu$. 
An $R(\nu)$-module structure on a vector space $V$ is the same as a homomorphism from $R(\nu)$ to $\End_k(V)$. If we precompose with the automorphism $a$ of $R(\nu)$, we get a new $R(\nu)$-module structure on the same vector space. This autoequivalence of the category of $R(\nu)$-modules is denoted $a^*$.

Let $\C_\nu$ be the category whose objects are pairs $(M,\sigma)$ where $M$ is a representation of $R(\nu)$ and $\sigma\map{a^*M}{M}$ is an isomorphism such that
\begin{equation}\label{defcnu}
 \sigma\circ a^*\sigma \circ \cdots \circ(a^*)^{n-1}\sigma = \mathrm{id}_M.
\end{equation}
 A morphism from $(M,\sigma)$ to $(M',\sigma')$ in $\C_\nu$ is a $R(\nu)$-module map $f\map{M}{M'}$ such that the following diagram commutes:
 \[
  \begin{CD}
   a^*M @>a^*f>> a^*M'\\
   @VV\sigma V   @VV\sigma' V\\
   M @>f>> M'.
  \end{CD}
 \]

 The following easy observation gives an alternative way to think about the category $\C_\nu$.
 \begin{lemma}\label{lemma:easy}
 The category $\C_\nu$ is equivalent to the category of graded representations of the smash product $R(\nu)\#\Z/n$.
\end{lemma}

 Let $\P_\nu$ be the full subcategory of finitely generated projective objects in $\C_\nu$ (an object of $\C_\nu$ is finitely generated if it is finitely generated as a module for $R(\nu)\#\Z/n$, equivalently if the underlying $R(\nu)$-module is finitely generated). The categorically minded reader may wish to parse this as the category of compact projective objects of $\C_\nu$.

 Let $\L_\nu$ be the full subcategory of $\C_\nu$ whose objects are pairs $(M,\sigma)$ where $M$ is finite dimensional, or equivalently the full subcategory of finite length objects. We have the following obvious corollary of Lemma \ref{lemma:easy}:

\begin{corollary}
 The categories $\C_\nu$ and $\L_\nu$ are abelian.
\end{corollary}

\section{The Grothendieck Group Construction}

Let $\zeta_n$ be a primitive $n$-th root of unity in $\mathbb{C}$.
Fix once and for all a ring homomorphism $\Z[\zeta_n]\to k$.

If $(M,\sigma)$ is an object of $\C_\nu$, then so is $\zeta_n(M,\sigma):=(M,\zeta_n\sigma)$.

An object $(A,\phi)$ of $\C_\nu$ is said to be \emph{traceless} if there is a representation $M$ of $R(\nu)$, an integer $t\geq 2$ dividing $n$ such that $(a^*)^tM\cong M$, and an isomorphism
\[
 A\cong M\oplus a^*M \oplus \cdots \oplus (a^*)^{t-1}M
\]
under which $\phi$ corresponds to an isomorphism carrying the summand $(a^*)^jM$ onto $(a^*)^jM$ for $1\leq j<t$ and the summand $(a^*)^tM$ onto $M$.

The group $K(\P_\nu)$ is defined to be the $\Z[\zeta_n]$-module generated by symbols $[(M,\sigma)]$ where $(M,\sigma)$ is an object of $\P_\nu$, subject to the relations
\begin{align*}
 [X] &= [X']+[X''] &  \text{if $X\cong X'\oplus X''$}\\
 [(M,\zeta_n\sigma)]&=\zeta_n[(M,\sigma)] & \\
 [X] &= 0 &\text{if $X$ is traceless}
\end{align*}

The group $K(\L_\nu)$ is defined similarly from the category $\L_\nu$, except that in place of the first relation, we have
\[
 [X]=[X']+[X'']
\]
if there is a short exact sequence $0\to X'\to X \to X''\to 0$.

Both $K(\P_\nu)$ and $K(L_\nu)$ are $\Z[\zeta_n,q,q\inv]$-modules, where $q$ acts by a grading shift.

Given two objects $(M,\sigma)$ and $(N,\tau)$ in $\C_\nu$, there is an induced automorphism of $\Hom_{R(\nu)}(M,N)$, namely
\[
 f\mapsto (a^*)\inv (\tau^{-1}\circ f \circ \sigma).
\]
We will call this automorphism $a_{\sigma\tau}$. It allows us to define a pairing in the following Lemma, whose proof is straightforward and omitted.

\begin{lemma}
The form defined by the below displayed equation descends to a semilinear pairing $K(\P_\nu)\times K(\L_\nu) \to \Z[\zeta_n]((q))$
\begin{equation}\label{pairing}
 \langle [(M,\sigma)],[(N,\tau)]\rangle =\sum_{d\in\Z}\mathrm{tr}(a_{\sigma\tau},\Hom_{R(\nu)}(M,N)_d)q^d
\end{equation}
\end{lemma}

Let $\psi$ be the antiautomorphism of $R(\nu)$ which sends each of the generators $e_\ii$, $y_j$ and $\tau_k$ to themselves.
If $P$ is a $R(\nu)$-module, we denote by $P^\psi$ the right $R(\nu)$-module whose underlying space is $P$ and the action of $R(\nu)$ is given by $p\cdot r=\psi(r)p$ for $p\in P$ and $r\in R(\nu)$.

Now given isomorphisms $\sigma\map{a^*P}{P}$ and $\tau\map{a^*Q}{Q}$, there is an induced isomorphism $\tau\otimes\sigma\map{a^*Q^\psi\otimes_{R(\nu)} a^*P}{Q^\psi\otimes_{R(\nu)} P}$. Since $a^*Q^\psi\otimes_{R(\nu)}a^*P$ is canonically isomorphic to $Q^\psi\otimes_{R(\nu)}P$, we can consider $\tau\otimes\sigma$ as an automorphism of $Q^\psi\otimes_{R(\nu)}P$. Using $\tau\otimes \sigma$, we also obtain a pairing of interest, whose existence proof is again straightforward and omitted.

\begin{lemma}
The form defined by the below displayed equation descends to a symmetric bilinear pairing $K(\P_\nu)\times K(\P_\nu)\to \Z[\zeta_n]((q))$.
\begin{equation}\label{projpairing}
 ([(P,\sigma)],[(Q,\tau)])=\sum_{d\in \Z}\mathrm{tr}(\tau\otimes \sigma,(Q^\psi\otimes_{R(\nu)}P)_d)q^d.
\end{equation}
\end{lemma}

\begin{remark}
 There is also a semilinear pairing $\langle\cdot,\cdot\rangle$ on $K(\P_\nu)$ defined using the $\Hom$ form using the same formula (\ref{pairing}). It is related to this bilinear pairing by $(x,y)=\langle x,\bar{y} \rangle$.
\end{remark}

%

\begin{lemma}\label{equivariantstructure}
 Let $L$ be a simple $R(\nu)$-module. Let $t$ be the minimal integer such that $(a^*)^tL\cong L$. Then there exists  $\sigma$ such that $(L\oplus a^*L \oplus\cdots \oplus (a^*)^{t-1}L,\sigma)$ is an object of $\C_\nu$.
\end{lemma}

\begin{remark}
 As the proof will indicate, this lemma and Theorem \ref{classifysimples} below are very general results about representations of an algebra with an action of a finite cyclic group. In particular, the same statements hold when $R(\nu)$ is replaced by $R(\la)\otimes R(\mu)$.
\end{remark}

\begin{proof}
Let $\sigma_0$ be an isomorphism from $(a^*)^tL$ to $L$. We have the canonical identification
\[
 a^*(L\oplus a^* L\oplus\cdots \oplus  (a^*)^{t-1}L)\cong a^*L\oplus (a^*)^2L\oplus\cdots\oplus (a^*)^tL.
\]
In block form, choose $\sigma$ to be of the form
$$
\begin{pmatrix}
 0  & 0 & \cdots & 0 & \la \sigma_0 \\
 1   & 0 & \cdots & 0 & 0 \\
 0  & 1 & \cdots & 0 & 0 \\
 \vdots & \vdots & \ddots  & \vdots & \vdots \\
0 & 0  & \cdots & 1 & 0 \end{pmatrix}
$$ where $\la\in k$.
In order for (\ref{defcnu}) to hold, we must have $$\la^{n/t}\sigma_0\circ (a^*)^t\sigma_0\circ\cdots\circ (a^*)^{n-t}\sigma_0=\id_L.$$ By Schur's Lemma $\sigma_0\circ (a^*)^t\sigma_0\circ\cdots\circ(a^*)^{n-t}\sigma_0$ is a nonzero scalar. Since $k$ is algebraically closed, there exists $\la$ satisfying this equation, completing the proof.
\end{proof}

It is clear that all modules formed in this way are simple in $\C_\nu$ and for different choices of $\la$ that they produce pairwise nonisomorphic simples.

 \begin{theorem}\label{classifysimples}
The construction of Lemma \ref{equivariantstructure} provides a classification of all simple objects in $\C_\nu$. In particular,
 a simple object of $\C_\nu$ is either traceless or of the form $(L,\sigma)$ with $L$ a simple $R(\nu)$-module.
\end{theorem}

\begin{proof}
We use the Induction-Restriction adjunction for the inclusion of algebras $R(\nu)\hookrightarrow R(\nu)\#\Z/n$. Suppose that $(S,\sigma)$ is a simple object in $\C_\nu$. Let $L$ be a simple $R(\nu)$-submodule of $S$. Then by adjunction, there is a morphism in $\C_\nu$ from $\Ind_{R(\nu)}^{\C_\nu}L$ to $(S,\sigma)$, which is surjective as $(S,\sigma)$ is simple.

The module $\Ind_{R(\nu)}^{\C_\nu}(L)$ is
\[
 (L\oplus a^*L \oplus\cdots \oplus (a^*)^{n-1}L,\phi)
\]
where $\phi$ is the obvious permutation matrix.

Now suppose that $t$ is the minimal integer such that $(a^*)^tL\cong L$. Let $\sigma_0$ and $\la$ be as in the proof of Lemma \ref{equivariantstructure} and write $X(L,\la,\sigma_0)$ for the corresponding module constructed in the proof. Define $\phi\map{\Ind_{R(\nu)}^{\C_\nu}(L)}{X(L,\la,\sigma_0)}$ by
\[
 \phi|_{(a^*)^{rt+s}L}=\la^{r}\sigma_0 a^*\sigma_0\cdots (a^*)^{r-1}\sigma_0\map{(a^*)^{rt+s}L}{(a^*)^sL}.
\]
It is easy to see that this does indeed define a morphism in $\C_\nu$. As $\la$ varies amongst all possible choices, we get different nonisomorphic simple quotients of $\Ind_{R(\nu)}^{\C_\nu}(L)$. A dimension count shows that there is a direct sum decomposition
\[
 \Ind_{R(\nu)}^{\C_\nu}(L)\cong \bigoplus_\la X(L,\la,\sigma_0).
\]
Therefore $(S,\sigma)$ is one of these direct summands, completing the proof.
\end{proof}

\section{Induction and Restriction}

For $\la,\mu\in\N J$, the automorphisms $a$ of $R(\la)$ and $R(\mu)$ induce an automorphism $a$ of $R(\la)\otimes R(\mu)$ by $a(v\otimes w)=av\otimes aw$. Let $\C_{\la\sqcup\mu}$ denote the category of graded representations of $(R(\la)\otimes R(\mu))\#\Z/n$. We define $\P_{\la\sqcup\mu}$ and $\L_{\la\sqcup\mu}$ in an analogous fashion.
 
 Given $(M,\sigma)$ and $(N,\tau)$ in $\C_\la$ and $\C_\mu$ respectively, we can form the induced module
 \[
  M\circ N:= R(\la+\mu)\bigotimes_{R(\la)\otimes R(\mu)} M\otimes N.
 \]
 The isomorphisms $\sigma$ and $\tau$ induce an isomorphism $a^*M\circ a^*N\to M\circ N$. When precomposed with the natural isomorphism $a^*(M\circ N)\cong a^*M\circ a^*N$, we obtain an isomorphism
 \[
  \sigma\circ \tau: a^*(M\circ N)\to M\circ N.
 \]
The object
\[
 (M,\sigma)\circ (N,\tau):= (M\circ N,\sigma\circ \tau)
\] is an element of $\C_{\la+\mu}$.

For $\la,\mu\in \N J$, let $e_{\la\mu}$ be the image of the identity under the inclusion $R(\la)\otimes R(\mu)\to R(\la+\mu)$. Given a $R(\la+\mu)$-module $M$, its restriction is defined by
\[
 \Res_{\la,\mu} M := e_{\la\mu} M.
\] It is a $R(\la)\otimes R(\mu)$-module.

Since $e_{\la\mu}$ is invariant under $a$, there is a canonical isomorphism $a^*(\Res M)\cong \Res(a^*M)$. Thus we obtain a restriction functor from $\C_{\la+\mu}$ to $\C_{\la\sqcup\mu}$.
%

Given objects $(M,\sigma)$ and $(N,\tau)$ of $\C_\la$ and $\C_\mu$ respectively, there is a tensor product object $(M\otimes N,\sigma\otimes \tau)$ of $\C_{\la\sqcup\mu}$.

\begin{proposition}\label{4.2}
 The tensor product induces isomorphisms of $\Z[\zeta_n,q,q\inv]$-modules
\begin{align*}
 K(\P_\la)\otimes_{\Z[\zeta_n]}K(\P_\mu)&\cong K(\P_{\la\sqcup\mu}), \\
 K(\L_\la)\otimes_{\Z[\zeta_n]}K(\L_\mu)&\cong K(\L_{\la\sqcup\mu}).
\end{align*}
\end{proposition}

\begin{proof}
 Since $k$ is algebraically closed, every indecomposable projective module for $R(\la)\otimes R(\mu)$ is a tensor product of projective modules over $R(\la)$ and $R(\mu)$. Furthermore, every simple module for $R(\la)\otimes R(\mu)$ is a tensor product of simple modules over $R(\la)$ and $R(\mu)$. Theorem \ref{classifysimples} now completes the proof.
\end{proof}

\begin{corollary}
 The restriction functor from $\C_{\la+\mu}$ to $\C_{\la\sqcup\mu}$ induces a coassociative coproduct on the direct sums
\[
 \bigoplus_{\nu\in \N J}K(\P_\nu)\quad \text{and}\quad \bigoplus_{\nu\in\N J} K(\L_\nu).
\]
\end{corollary}

\begin{theorem}
 The functors of induction and restriction form an adjoint pair of exact functors between $\C_{\la\sqcup\mu}$ and $\C_{\la+\mu}$.
\end{theorem}

\begin{proof}
The adjunction follows from the usual tensor-hom adjunction.
 Exactness follows from exactness in the situation where there is no automorphism $a$. For restriction, this is obvious and for induction, exactness follows from \cite[Proposition 2.16]{khovanovlauda}.
\end{proof}

\begin{definition}
 A functor is said to be traceless if its image lies in the full subcategory of traceless objects.
\end{definition}

We refer to the filtration appearing in the below theorem as the \emph{Mackey filtration}. This important Theorem in the unfolded case is \cite[Proposition 2.18]{khovanovlauda} 

\begin{theorem}\label{mackey}
Let $\la_1,\ldots,\la_k,\mu_1\ldots,\mu_l\in\N J$ be such that $\sum_i \la_i=\sum_j \mu_j$. Consider the exact functor
$$ \Res_{\mu_1,\ldots,\mu_l}\circ\Ind_{\la_1,\ldots,\la_k}
$$ from $\C_{\la_1\sqcup\cdots\sqcup\la_k}$ to $\C_{\mu_1\sqcup\cdots\sqcup\mu_l}$ This composite functor has a natural filtration by exact functors. The subquotient functors in this filtration which are not traceless are indexed by tuples $\nu_{ij}$ satisfying $\la_i=\sum_j \nu_{ij}$ and $\mu_j=\sum_i\nu_{ij}$,
and are isomorphic, up to a grading shift, to the composition $$\Ind_\nu^\mu\circ \tau\circ \Res_\nu^\la.$$ Here
$\Res_{\nu}^\la\map{\otimes_i R(\la_i)\mods}{\otimes_i(\otimes_jR(\nu_{ij}))\mods}$ is the tensor product of the $\Res_{\nu_{i\bullet}}$, $\tau\map{\otimes_i(\otimes_jR(\nu_{ij}))\mods}{\otimes_j(\otimes_iR(\nu_{ij}))\mods}$ is given by permuting the tensor factors and $\Ind_\nu^\mu\map{\otimes_j(\otimes_iR(\nu_{ij}))\mods}{\otimes_j R(\mu_j)\mods}$ is the tensor product of the $\Ind_{\nu_{\bullet i}}$.
\end{theorem}

\begin{proof}
For simplicity we give the proof for the case where $k=l=2$, from which the general case follows. Write $B$ for the $(R(\mu_1)\otimes R(\mu_2),R(\la_1)\otimes R(\la_2))$-bimodule $R(\la_1+\la_2)$. Then tensoring with $B$ is the usual composition $\Res_{\mu_1,\mu_2}\otimes \Ind_{\la_1,\la_2}$ without the automorphism $a$ considered. We now describe the filtration of $B$ as in the proof of \cite[Proposition 2.18]{khovanovlauda}.

Let  $ \eta_\bullet=(\eta_{11},\eta_{12},\eta_{21},\eta_{22})$ be a quadruple of elements in $\N I$ such that $\eta_{i1}+\eta_{i2}=\la_i$ and $\eta_{1j}+\eta_{2j}=\mu_j$ for $i,j\in\{1,2\}$. Let $w(\eta_\bullet)$ be the involutive permutation where $w(i)=i$ if $i\leq |\eta_{11}|$ or $i>|\la_1+\eta_{21}|$ and $w(i)=i+|\eta_{12}|$ if $|\eta_{11}|<i\leq |\la_1|$. Define the element $u(\eta_\bullet)=\tau_{w(\eta_\bullet)}\in R(\la_1+\la_2)$.

Then $B$ is generated as a bimodule by the elements $u(\eta_\bullet)$. Define a partial order on the set of all quadruples $\eta_\bullet$ by $\eta_\bullet\geq\eta_\bullet'$ if $\eta_{21}-\eta_{21}'\in \N I$. Let $B(\leq\eta_\bullet)$, respectively $B(<\eta_\bullet)$ be the subbimodule of $B$ generated by all $u(\zeta_\bullet)$ with $\zeta_\bullet\leq \eta_\bullet$, respectively $\zeta_\bullet<\eta_\bullet$. Let $B(\eta_\bullet)=B(\leq\eta_\bullet)/B(<\eta_\bullet)$.

The functor of tensoring with the subquotient $B(\eta_\bullet)$ is the composition $\Ind_\nu^\mu\circ \tau\circ \Res_\nu^\la$. In particular it is exact, so we indeed have a filtration by exact subfunctors. We now need to understand how this filtration interacts with the automorphism $a$.

If $a\eta_\bullet=\eta_\bullet$, then the functor $B(\eta_\bullet)\otimes -$ coincides with the composite functor $\Ind_\nu^\mu\circ \tau\circ \Res_\nu^\la$ from $\C_{\la_1\sqcup\la_2}$ to $C_{\mu_1\sqcup\mu_2}$.

If $a\eta_\bullet\neq\eta_\bullet$, consider $B(\eta_\bullet)\oplus B(a\eta_\bullet)\oplus\cdots B(a^{t-1}\eta_\bullet)$ where $t$ is the minimal positive integer such that $a^t\eta_\bullet=\eta_\bullet$. Grouping the subquotients in orbits in this manner produces a subquotient functor which is tensoring with the direct sum bimodule. Furthermore, this subquotient functor is traceless, because the $a$-action comes from the $a$-action on the bimodule which is permuting the summands.
\end{proof}

\begin{theorem}
The induction and restriction functors send projective objects to projective objects.
\end{theorem}

\begin{proof}
The induction and restriction functors have right adjoints restriction and coinduction respectively. This coinduction functor is discussed in \cite[\S 2.3]{laudavazirani} in the unfolded case with a straightforward generalisation to the folded case. The restriction functor is clearly exact and the coinduction functor is exact since $R(\la+\mu)$ is a free $R(\la)\otimes R(\mu)$-module. The theorem follows since a  functor sends projectives to projectives if it has an exact right adjoint.
\end{proof}

\section{Duality}

Let $P$ be a finitely generated projective $R(\nu)$-module. Then we can define a dual module $$\D P=\Hom(P,R(\nu)).$$

This is the direct sum of all homogeneous homomorphisms between $P$ and $R(\nu)$, not necessarily of zero degree. It is a graded $k$-vector space where $(\D P)_m= \Hom(q^mP,R(\nu))_0$, where we are referring to degree zero homomorphisms only.

This is also a $R(\nu)$-module, where the action of $R(\nu)$ is by
\begin{equation}\label{dualaction}
 r(\la)(m)=\la(\psi(r)m)
\end{equation}
for all $r\in R(\nu)$, $\la\in\D(P)$ and $m\in P$.

Let $M$ be a finite dimensional $R(\nu)$-module. Then its dual $\D(M):=\Hom_k(M,k)$ is also an $R(\nu)$-module by
the same formula (\ref{dualaction}) for all $r\in R(\nu)$, $\la\in\D(M)$ and $m\in M$.
The module $\D M$ is also naturally graded.

For any morphism $f\map{M}{N}$ between finitely generated projective, or finite dimensional $R(\nu)$-modules, there is then an induced morphism $\D f\map{\D N}{\D M}$.

Now we define the dual of an object $(M,\sigma)$ in $\L_\nu$ or $\P_\nu$ by the formula
\[
 \D(M,\sigma)=(\D M,(\D\sigma)\inv)
\]
where the appropriate duality is taken depending on which category we are in. It is clear that $\D$ is a contravariant functor from $\L_\nu$ to itself. To show that $\D$ sends $\P_\nu$ to $\P_\nu$, since projective modules are direct summands of free modules, it suffices to show that dual of a free $R(\nu)\#\Z/n$-module is free. The free $R(\nu)\#\Z/n$-module of rank one is
\[
 (R(\nu)\oplus a^*R(\nu)\oplus \cdots\oplus (a^*)^{n-1}R(\nu),\phi)
\]
where, $\phi$ is the obvious permutation matrix. It is straightforward to compute that this is isomorphic to its dual.

An object $M$ of $\L_\nu$ or $P$ of $\P_\nu$ is said to be self-dual if there is an isomorphism $\D M\cong M$ in $\L_\nu$, respectively $\D P\cong P$ in $\P_\nu$.

\begin{lemma}\label{lem:dualproj}
 If $P$ and $Q$ are projective then $\D(P\circ Q)\cong \D P\circ \D Q$.
\end{lemma}

\begin{proof}
Since projective modules are all direct summands of free modules, 
it suffices to consider the case where $P$ and $Q$ are free. This computation is similar to the one which shows that $\D$ sends $\P_\nu$ to $\P_\nu$.
\end{proof}


%
%

\section{The Statement of the Categorification Theorems}

The algebra $\f_{\Q(q)}$ is the $\Q(q)$-algebra as defined in \cite{lusztigbook} generated by elements $\{\th_j\mid j\in J\}$. Lusztig defines it as the quotient of a free algebra by the radical of a bilinear form. By the quantum Gabber-Kac theorem, it has a presentation as the quotient of a free algebra by the quantum Serre relations.
Morally, $\f_{\Q(q)}$ should be thought of as the positive part of the quantised enveloping algebra $U_q(\g)$.
There is only a slight difference in the coproduct, necessary as the coproduct of $U_q(\g)$ does not map $U_q(\g)^+$ into $U_q(\g)^+\otimes U_q(\g)^+$.

There is a $\Z[q,q\inv]$-form of $\f_{\Q(q)}$, which we denote simply by $\f$. It is the $\Z[q,q\inv]$-subalgebra of $\f_{\Q(q)}$ generated by the divided powers $\th_j^{(n)}=\th_j^n/[n]_{q_j}!$, where $[n]_q!=\prod_{i=1}^n(q^i-q^{-i})/(q-q\inv)$ is the $q$-factorial and $q_j=q^{d_j}$. 

The algebra $\f$ is graded by $\N J$ where $\th_j$ has degree $j$ for all $j\in J$.
We write $\f=\oplus_{\nu\in\N J} \f_\nu$ for its decomposition into graded components.
Of significant importance for us is the dimension formula from \cite[Theorem 33.1.3]{lusztigbook}.
\begin{equation*}
\sum_{\nu\in \N J}\dim \f_\nu t^\nu = \prod_{\a\in\Phi^+} (1-t^{\a})^{-\operatorname{mult}(\a)}.
\end{equation*}
where $\Phi^+$ is the set of positive roots in the root system defined by $(J,\cdot)$.

The tensor product $\f\otimes \f$ has an algebra structure given by
\begin{equation}\label{twistedproduct}
 (x_1\otimes y_1)(x_2\otimes y_2)=q^{\b_1\cdot\a_2}x_1 x_2 \otimes y_1 y_2
\end{equation}
 where $y_1$ and $x_2$ are homogeneous of degree $\b_1$ and $\a_2$ respectively.

Given any bilinear form $(\cdot,\cdot)$ on $\f$, we obtain a bilinear form $(\cdot,\cdot)$ on $\f\otimes \f$ by
\[
 (x_1 \otimes x_2,y_1\otimes y_2)=(x_1,x_2)(y_1,y_2).
\]

There is a unique algebra homomorphism $r\map{\f}{\f\otimes \f}$ such that $r(\th_j)=\th_j\otimes 1 + 1 \otimes \th_j$ for all $j\in J$. For each $x\in\f_\nu$, define $_jr(x)$ and $r_j(x)$ where $r(x)={_j}r(x)\otimes \th_j$ plus terms in other bidegrees and $r(x)=\th_j\otimes r_j(x)$ plus terms in other bidegrees.

The algebra $\f$ has a symmetric bilinear form $\langle\cdot,\cdot\rangle$ satisfying
\begin{eqnarray*}
 \langle\th_j,\th_j\rangle&=&(1-q^{j\cdot j})^{-1} \\
\langle x,yz\rangle &=& \langle r(x),y\otimes z\rangle.
\end{eqnarray*}

The form $\langle \cdot,\cdot\rangle$ is nondegenerate. Indeed, in the definition of $\f$ in \cite{lusztigbook}, $\f_{\Q(q)}$ is defined to be the quotient of a free algebra by the radical of this bilinear form. It is known that $\f$ is a free $\Z[q,q\inv]$-module. Let $\f^*$ be the graded dual of $\f$ with respect to $\langle \cdot,\cdot\rangle$. By definition, $\f^*=\oplus_{\nu\in\N I} \f_\nu^*$. As twisted bialgebras over $\Q(q)$, $\f_{\Q(q)}$ and $\f^*_{\Q(q)}$ are isomorphic, though there is no such isomorphism between their integral forms.

There is a bar involution on $\f$ which is the algebra automorphism of $\f$ sending $q$ to $q\inv$ and fixing the generators $\th_j$. It induces a bar involution on $\f^*$ by $\overline{\la}(x)=\la(\overline{x})$ for all $\la\in\f^*$ and $x\in \f$.

Let $A$ denote the smallest subring of $\Z[\zeta_n]$ such that the structure constants for multiplication in $K(\P)$, with respect to the self-dual indecomposable projectives, lie in $A[q,q\inv]$. We show later in Lemma \ref{mstructureconstants} that $A\subset \Z[\zeta_n+\zeta_n\inv]$ and in the most favourable circumstances, $A=\Z$. This implies that if $A\not\cong \Z$ then $n\geq 5$ and it is known that every irreducible Cartan datum of finite or affine type can be obtained via a folding with $n<5$.

Let $\kk_\nu$ denote the $A[q,q\inv]$-span of the classes of self-dual indecomposable projectives inside $K(\P_\nu)$. Similarly, let $\kk_\nu^*$ denote the $A[q,q\inv]$-span of the classes of self-dual simple modules inside $K(\L_\nu)$. Write $\kk=\oplus_{\nu\in \N J}\kk_\nu$ and $\kk^*=\oplus_{\nu\in\N J}\kk_\nu^*$ for the decomposition of $\kk$ and $\kk^*$ into their graded pieces.

We are now able to state the main categorification theorems of this paper. They will be proved at the end of Section \ref{proof}.

\begin{theorem}\label{projiso}
There is a unique $A[q,q\inv]$-linear grading preserving isomorphism $\ga\map{\f\otimes A}{\kk}$ such that
\begin{enumerate}
 \item $\ga(\th_j^{(m)})=[P_j^{(m)}]$ for all $j\in J$ and $m\in \N$.
\item Under the isomorphism $\ga$, the multiplication $\f_\la\otimes \f_\mu\to\f_{\la+\mu}$ corresponds to the product on $\kk$ induced by $\Ind_{\la,\mu}$.
\item Under the the isomorphism $\ga$, the comultiplication $\f_{\la+\mu}\to \f_\la\otimes \f_\mu$ corresponds to the  coproduct on $\kk$ induced by $\Res_{\la,\mu}$.
\item Under the isomorphism $\ga$, the bar involution on $\f$ corresponds to the anti-linear involutive automorphism on $\kk$ induced by the duality $\D$.
\item The isomorphism $\ga$ intertwines the pairings $\langle\cdot,\cdot\rangle$ and $(\cdot,\cdot)$ defined on $\f$ and $\kk$ respectively.
\end{enumerate}
 \end{theorem}

\begin{theorem}\label{simiso}
 There is a unique $A[q,q\inv]$-linear grading preserving isomorphism $\ga^*\map{\kk^*}{\f^*\otimes A}$ such that
\begin{enumerate}
 \item $\ga^*([L(j)])=\theta_j^*$ for all $j\in J$.
\item Under the isomorphism $\ga^*$, the multiplication $\f^*_\la\otimes \f^*_\mu\to \f^*_{\la+\mu}$ corresponds to the product on $\kk^*$ induced by $\Ind_{\la,\mu}$.
\item Under the the isomorphism $\ga^*$, the comultiplication $\f^*_{\la+\mu}\to \f^*_\la\otimes \f^*_\mu$ corresponds to the  coproduct on $\kk^*$ induced by $\Res_{\la,\mu}$.
\item Under the isomorphism $\ga^*$, the bar involution on $\f^*$ corresponds to the anti-linear antiautomorphism on $\kk^*$ induced by the duality $\D$.
\item The isomorphism $\ga^*$ is the graded dual of the isomorphism $\ga$ in Theorem \ref{projiso}.
\end{enumerate}
\end{theorem}

\section{One Colour Folded KLR Algebras}\label{onecolour}

Let $j\in J$. Let the elements of the orbit $j$ be $i_1,i_2,\ldots,i_t$.

The algebra $R(i_1+\cdots+i_t)$ has a unique irreducible representation $L_j$. This representation has a basis $[w]$ where $w$ runs over all permutations of $\{i_1,\ldots,i_t\}$. The idempotent $e_\ii$ acts on $[w]$ by 1 if $\ii=w$ and by 0 otherwise. The $y_i$ all act by zero. The $\tau_i$ act by the symmetric group action on the set of all permutations.

The module $a^*L_j$ has by construction the same underlying vector space but the action is twisted by the automorphism $a$. We define $\sigma\map{a^*L_j}{L_j}$ by $\sigma[w]=[aw]$. Define
\[
 L(j)=(L_j,\sigma).
\]
Let $P(j)$ be the projective cover of $L(j)$ in $\C_j$. We can explicitly construct $P(j)$ in a similar fashion to the construction of $L(j)$. It is of the form $P(j)=(P_j,\sigma)$ where $P_j=\oplus_{w\in S_t} k[x_1,\ldots x_t][w]$.

\begin{lemma}\label{7.1}
 For all $j\in J$, we have $(P(j),P(j))=(1-q^{j\cdot j})\inv$.
\end{lemma}

\begin{proof}
Let $t=|j|$. Recall that as a vector space, $P_j=\oplus_{w\in S_t} k[x_1,\ldots,x_t][w]$. There is an isomorphism
\[
 P_j^\psi \otimes_{R(\nu)} P_j\cong k[z_1,\ldots,z_t]
\]
given by
\[
 f(x)[w]\otimes g(y)[u]\mapsto \delta_{wu} f(z)g(z).
\]
The automorphism $\sigma\otimes\sigma$ acts by permuting the $z_i$'s. In the monomial basis, this action is by a permutation matrix, so the trace is the number of fixed points, which consists only of powers of $z_1z_2\cdots z_t$. The statement of the lemma follows.
\end{proof}

\begin{lemma}
 The object $L(j^n):=q_j^{{n \choose 2}}L(j)^{\circ n}$ is a self-dual simple object of $\C_{nj}$.
\end{lemma}

\begin{proof}
 The object $L(j^n)$ is irreducible because $L_j^{\circ n}$ is an irreducible $R(n(i_1+\cdots+i_t))$-module. The self-duality appears later in a more general form in Lemma \ref{fselfdual}. We don't bother to prove it here as it is not yet important.
\end{proof}

Let $P(j)^{(n)}$ be the projective cover of $L(j^n)$ in $\C_j$

\begin{theorem}\label{7.3}
 In $\displaystyle\bigoplus_{n=0}^\infty K(\P_{nj})$, we have the identity $$[P(j)^{(m)}][P(j)^{(n)}]={m+n \brack n}_j [P(j)^{(m+n)}].$$
\end{theorem}

\begin{proof}
It suffices to prove the dual version, namely
\begin{equation}\label{onecoloureqn}
 r([L(j^n)])=\sum_{a+b=n}{n\brack a}_j [L(j^a)]\otimes [L(j^b)].
\end{equation}

 Consider $\Res_{mj,nj} L(j^{m+n})$ under the Mackey filtration. All terms either come from permutations or are trivially traceless by the $a$-action on the Mackey filtration.
\end{proof}

\section{Some Important Lemmas}

This section is modelled on \cite[\S 3.2]{khovanovlauda}, the technical heart of that paper, which in turn is modelled on \cite{grojnowski}, which \cite[Ch 5]{kleshchevbook} is an exposition of in the case of graded Hecke algebras. We provide all proofs for the sake of completeness.

For any $j\in J$ and any object $M$ of $\C_\nu$, define
\begin{equation}\label{defepj}
 \epsilon_j(M) = \{\operatorname{max} n\in \N \mid \Res_{nj,\nu-nj}M\neq 0\}.
\end{equation}

%
%
%

\begin{lemma}\label{3.7}
 Let $N$ be an irreducible object in $\C_{\nu}$ and $j\in J$ such that $\epsilon_j(N)=0$. Let $M=L(j)^{\circ n}\circ N$. Then
\begin{enumerate}
 \item $\Res_{nj,\nu-nj}\cong L(j^n)\otimes N$.
\item The head of M is irreducible and $\epsilon_j(\operatorname{hd}M)=n$
\item All other composition factors $L$ of $M$ have $\epsilon_j(L)<\epsilon_j(M)$.
\end{enumerate}
\end{lemma}

\begin{proof}
 Since $\e_j(N)=0$, (1) is a consequence of the Mackey filtration.
 
 If $Q$ is a quotient of $M$, then by adjunction there is a morphism from $L(j^n)\otimes N$ to $\Res_{nj,\nu-nj}Q$, which is injective as $N$ is irreducible.
 Now suppose that $Q_1\oplus Q_2$ is a quotient of $M$ with $Q_1$ and $Q_2$ nonzero. As restriction is exact, we get a surjection from $\Res_{nj,\nu-nj}M$ to $\Res_{nj,\nu-nj}Q_1\oplus \Res_{nj,\nu-nj}Q_2$.
 By (1) this is a contradition, hence the head of $M$ is irreducible. Furthermore this argument also shows that $\e_j(\hd(M))=n$.
 
Now consider the short exact sequence $0\to K\to M\to \hd(M)\to 0$ and apply $\Res_{nj,\nu-nj}$. We have shown above that there is an induced isomorphism $\Res_{nj,\nu-nj}M\cong \Res_{nj,\nu-nj}\hd(M)$. Therefore $\Res_{nj,\nu-nj}K=0$ which proves (3).
\end{proof}

\begin{lemma}\label{3.8}
 Let $M$ be irreducible in $\C_\nu$ and $j\in J$. Let $m=\epsilon_j(M)$. Then there exists a simple object $X$ in $\C_{\nu-mj}$ such that
 \[
  \Res_{mj,\nu-mj}(M)\cong  L(j^m)\otimes X.
 \]
\end{lemma}

\begin{proof}
 Since $L(j)^{\circ m}$ is the only simple $R(mj)$-module, there exists a simple $X$ such that $L(j^m)\otimes X$ is a submodule of  $\Res_{mj,\nu-mj}(M)$. Since $m=\e_j(M)$ it must be that $\e_j(X)=0$. The inclusion of $L(j^m)\otimes X$ into $\Res_{mj,\nu-mj}(M)$ factors through the adjunction morphism in the following diagram:
 
 \centerline{\xymatrix{
&  L(j^m)\otimes X  \ar[rr]\ar[dr]_{a}
&  
  & \Res_{mj,\nu-mj}(M)
 \\
&  
  &  \Res_{mj,\nu-mj}(L(j^m)\circ X)\ar[ur]_f
  & 
}}

Since $\e_j(X)=0$, the Mackey filtration shows that the adjunction morphism $a$ is an isomorphism. 
The inclusion of $L(j^m)\otimes X$ into $\Res_{mj,\nu-mj}(M)$ yields by adjunction a morphism from $L(j^m)\circ X$ to $M$ which is surjective as $M$ is simple. The map $f$ is obtained by applying the restriction functor to this surjection. Since the restriction functor is exact, $f$ is surjective.

Therefore $f\circ a$ is a surjective map from a simple source to a target which is nonzero as $\e_j(M)=m$. Thus it is an isomorphism.
\end{proof}

\begin{lemma}\label{8.4}
 Let $N$ be irreducible in $\C_\nu$, $m\in \N$ and $j\in J$. Let $M=L(j^m)\circ N$. Then the head of $M$ is irreducible, $\e_j(\hd(M))=\e_j(N)+m$ and all other composition factors $L$ of $M$ have $\e_j(L)<\e_j(N)+m$.
\end{lemma}

\begin{proof}
 Let $n=\e(N)$. By Lemma \ref{3.8}, there exists a simple object $X$ such that
 \[
  \Res_{nj,\nu-nj}(N)\cong L(j^n)\otimes X.
 \]
This object $X$ satisfies $\e_j(X)=0$.

By adjunction, there is a nonzero morphism from $L(j^n)\circ X$ to $N$, which is surjective as $X$ is simple. Applying the exact functor $L(j^m)\circ -$, we obtain a surjection from $L(j^{m+n})\circ X$ to $M$. Therefore all composition factors of $M$ are composition factors of $L(j^{m+n})\circ X$. The statement of this lemma now follows from Lemma \ref{3.7}(2) and (3).
\end{proof}

\section{Crystal Operators}

Recall the definition of $\e_j(M)$ from (\ref{defepj}). We similarly define
\[
 \epsilon_j^*(M) = \{\operatorname{max} n\in \N \mid \Res_{\nu-nj,nj}M\neq 0\}.
\]

For any two modules $M$ and $N$, define $M\diamond N=\hd(M\circ N)$.

Let $\sB_\nu$ be set of all self-dual $(L,\sigma)$ in $\C_\nu$ where $L$ is simple.
Let $\sB=\sqcup_{\nu\in \N J} \sB_\nu$.
We now construct some crystal operators on $\sB$.

For $M\in \sB_\nu$, define
\begin{align*}
 \tilde e_j M &= q_j^{1-\epsilon_j(M)}\soc\Hom_{R(j)}(L(j),\Res_{j,\nu-j}M) \\
 \tilde e_j^*M &= q_j^{1-\epsilon_j^*(M)}\soc\Hom_{R(j)}(L(j),\Res_{\nu-j,j}M) \\
 \tilde f_j M &= q_j^{\epsilon_j(M)} L(j)\diamond M \\
 \tilde f_j^* M &= q_j^{\epsilon_j^*(M)} M\diamond L(j)\\
 \wt(M)&=\nu.
\end{align*}

Since $L(j)$ is the unique simple object in $\C_j$, the following identities hold:
\begin{align*}
 \epsilon_j(M) &= \max\{n\in\N \mid \tilde e_j^nM\neq 0\} \\
 \epsilon_j^*(M) &= \max\{n\in\N \mid (\tilde e_j^*)^nM\neq 0\} 
 \end{align*}

\begin{lemma}
 Assume that $\tilde e_j$ and $\tilde f_j$ send $\sB$ to $\sB\sqcup\{0\}$. Then for $b,b'\in \sB$, $b=\t f_j b$ if and only if $b'=\t e_j b$.
\end{lemma}

\begin{proof}
 This is immediate from the adjunction isomorphism
 \[
  \Hom_{\C_{\nu+j}}(L(j)\circ N,M)\cong \Hom_{\C_\nu}(\Hom_{R(j)} (L(j),\Res_{j,\nu-j} M).
 \]\qedhere\end{proof}

The rest of this section is dedicated to showing that the operators $\tilde e_j$ and $\tilde f_j$ send $\sB$ to $\sB\sqcup\{0\}$.


The KLR algebra $R(j)$ is symmetric in the sense of \cite[Definition 1.3]{kkko}. In particular there is a $R$-matrix $r_{L_j,M}\map{L_j\circ M}{M\circ L_j}$ for every $R(\nu)$-module $M$. This is a homogeneous morphism whose degree is given by the following lemma:

\begin{lemma}\label{degr}
If $M$ is simple, then the degree of $r_{L_j,M}$ is $(\nu-\e_j(M)j)\cdot j$.
\end{lemma}

\begin{proof}
  If $\e_j(M)=0$, then this is \cite[Proposition 10.1.3]{kkko2}. To deal with the general case, we induct on $\e_j(M)$. So suppose that $\e_j(M)>0$. Then we can write $M=\hd (L_j\circ N)$ for some simple $N$. By Lemma \ref{8.4}, $\e_j(N)<\e_j(M)$ so by induction we can assume the result known for $N$. Consider the diagram
  \[  \begin{CD}
   L_j\circ M @<<< L_j\circ L_j\circ N\\
   @VVr_{L_j,M} V   @VV\mathrm{id}\circ r_{L_j,N} V \\
   M\circ L_j @<<< L_j\circ N\circ L_j.
  \end{CD}
 \]
 This diagram commutes and the degrees of the horizontal maps are zero so the degrees of the vertical maps must agree.
 \end{proof}

\begin{lemma}\label{fselfdual}
 If $(M,\sigma)\in \sB_\nu$ then  $\tilde f_j (M,\sigma)\in\sB_{\nu+j}$.
\end{lemma}
%

\begin{proof}
 By \cite[Theorem 3.2]{kkko}, if $M$ is simple, then the image of the $R$-matrix can be identified with both the head of $L(j)\circ M$ and the socle of $M\circ L(j)$, and furthermore this image is simple. Consider the diagram:
%

\centerline{\xymatrix@C=0.31	cm{
&  L_j\circ M \ar@/^2.0pc/@[red][rrrrr]^{r_{L_j,M}} \ar@{->>}[r]
& L_j\diamond M \ar@{.>}[r]^-{X} 
  & \D(L_j\diamond M)  \ar@{^{(}->}[r] 
  & \D(L_j\circ M) \ar[r]
& \D M\circ \D L_j\ar[r] & M\circ L_j \\
& a^*(L_j\circ M) \ar@{->>}[r]\ar[u]\ar@/^-2.0pc/@[red][rrrrr]^{a^*r_{L_j,M}} 
  & a^*(L_j\diamond M)\ar@{.>}[r] \ar[u]
  & a^*\D(L_j\diamond M) \ar@{^{(}->}[r] \ar[u] 
& a^*\D(L_j\circ M)\ar[r] \ar[u] & a^*\D M\circ a^* \D L_j
\ar[r] \ar[u] & a^*(M\circ L_j) \ar[u]
}}

The unlabelled morphisms in the top row are, from left to right, the canonical surjection, the canonical inclusion, the canonical isomorphism, and the circle product of the morphisms $f\map{\D M}{M}$ and $g\map{\D L_j}{L_j}$ which exhibit isomorphisms $(M,\sigma)\cong (\D M,(\D\sigma)\inv)$ and $(L_j,\sigma_j)\cong (\D(L_j),(\D\sigma_j)\inv)$.
There exists a unique morphism $X\map{L_j\diamond M}{\D(L_j\diamond M})$ making the top row commute. Each morphism on the bottom row is the pullback of the morphism on the top row under $a^*$.

We have to show the commutativity of the square
\[  \begin{CD}
   L_j\diamond M @>X>> \D(L_j\diamond M)\\
   @AA\sigma' A   @AA\D(\sigma')\inv A  \\
   a^*(L_j\diamond M) @>a^*X>> a^*\D(L_j\diamond M).
  \end{CD}
 \]
This square commutes because all other squares in the large diagram above are already known to commute.

To check the grading shift, note that the degree of the isomorphism $\D(L_j\circ M)\cong \D M \circ \D L_j$ is $\nu\cdot j$. The degree of $r_{L_j,M}$ is $(\nu-\e_j(M)j)\cdot j$ by Lemma \ref{degr}. Therefore the degree of the isomorphism $\D(L_j\diamond M)\cong L_j \diamond M$ is $\e_j(M)j\cdot j$. Thus a degree shift of $q_j^{\e_j(M)}$ is required to get a self-dual object.
\end{proof}

\begin{lemma}
If $M\in \sB_\nu$ then $\tilde e_j (M)\in \sB_{\nu-j}\cup\{0\}$.
\end{lemma}

\begin{proof}
Assume that $\e_j(M)\neq 0$, so in particular $e_j^*(M)>0$. Let $N$ be a simple direct summand of $\e_j(M)$. Then by adjunction, $M$ is a quotient of $L(j)\circ N$. 
Let $m=\e_j(M)$. Then Lemma \ref{3.8} finds a simple $X$ such that
\[
 \Res_{mj,\nu-mj}(M)\cong L(j^m)\otimes X.
\]
As the restriction functor is exact, this is a quotient of $\Res_{mj,\nu-mj} (L(j)\circ N)$.

By Lemma \ref{8.4}, $\e_j(N)=m-1$. Therefore in the Mackey filtration of $\Res_{mj,\nu-mj} (L(j)\circ N)$, there is only one nonzero term, which involves considering $\Res_{(m-1)j,\nu-mj}N$. From Lemma \ref{3.8}, we have 
$$\Res_{(m-1)j,\nu-mj}N\cong L(j^{m-1})\otimes Y$$ for some simple $Y$.

We therefore have $X\cong Y$. Since $N$ can be recovered as the head of $L(j^{m-1})\circ Y$, this means that $N$ is uniquely determined by $M$.

To show that this simple summand $N$ has multiplicity one in $\tilde e_j (M)$ follows from the adjunction formula
\[
  \Hom_{\C_{\nu+j}}(L(j)\circ N,M)\cong \Hom_{\C_\nu}(\Hom_{R(j)} (L(j),\Res_{j,\nu-j} M).
 \]
Therefore $\tilde e_j (M)$ is simple. We hereafter denote it by $N$.

 By the classification of irreducibles in Theorem \ref{classifysimples}, there exists an isomorphism $N\cong \zeta \D N$ for some root of unity $\zeta$. The large commutative diagram in the proof of the previous Lemma now shows that as $M\cong \tilde f_i N$, we have $M\cong \zeta\D M$. Since $M$ is assumed self-dual, $\zeta=\pm 1$, completing the proof.
\end{proof}

\section{The Crystal}\label{sec:crystal}

We define
a crystal for the Cartan datum $(J,\cdot)$ to be a set $B$ together with maps $\tilde{e}_i,\tilde{f}_i\map{B}{B\sqcup\{0\}}$, $\e_i,\phi_i\map{B}{\Z}$ for all $i\in J$, and $\wt\map{B}{\Z J}$,  satisfying the conditions
\begin{enumerate}
 \item $\phi_i(b)=\e_i(b)+\langle h_i,\wt(b)\rangle$
 \item If $b\in B$ and $\tilde e_i(b)\in B$ then $\wt(\tilde e_i(b))=\wt(b)+\a_i$, $\e_i(\tilde e_ib)=\e_i(b)-1$ and $\phi_i(\tilde e_ib)=\phi_i(b)+1$.
 \item If $b\in B$ and $\tilde f_i(b)\in B$ then $\wt(\tilde f_i(b))=\wt(b)-\a_i$, $\e_i(\tilde f_ib)=\e_i(b)+1$ and $\phi_i(\tilde f_ib)=\phi_i(b)-1$.
 \item If $b,b'\in B$, then $b'=\tilde e_i b$ if and only if $b=\tilde f_ib'$.
\end{enumerate}

There are more general notions of a crystal in the literature which allow $\e_i$ and $\phi_i$ to take the value $-\infty$. Also the image of $\wt$ is usually allowed to land in the entire weight lattice, as opposed to $\Z J$. Since we do not come across these crystals here, we shall ignore them.

\begin{definition}
 A crystal $B$ is highest weight if
 \begin{enumerate}
  \item It has a distinguished element $b_0$ such that every $b\in B$ can be reached from $b_0$ by applying the operators $\t f_j$ for $j\in J$.
  \item For all $b\in B$ and $j\in J$, $\e_j(b)=\max\{n\mid \t e_j^n(b)\neq 0\}$.
 \end{enumerate}
\end{definition}

The following result allows us to identify a crystal as $B(\infty)$ (the crystal which is a combinatorial skeleton of a Verma module). It is equivalent to the criterion of Kashiwara and Saito \cite[Proposition 3.2.3]{kashiwarasaito}, and we sketch a proof of this equivalence later in this section.

\begin{theorem}\cite[Proposition 1.4]{tingleywebster}\label{criterion}
 Assume $(B,\tilde e_i,\tilde f_i)$ and $(B,\tilde e_i^*,\tilde f_i^*)$ are highest weight crystals with the same highest weight element $b_0$ of weight 0. Suppose that the following conditions are satisfied for all $b \in B$ and all distinct pairs $i,j\in J$:
 \begin{enumerate}
 \item 	$\t f_i(b)\neq 0$, $\t f^*_i(b)\neq 0$,
 \item $\t f_i \t f_j^*(b) = \t f_j^* \t f_i (b)$,
 \item $\e_i(b)+\e_i^*(b)+ \langle \operatorname{wt}(b),\a_i^\vee\rangle\geq 0$,
\item If $\e_i(b)+\e_i^*(b)+ \langle \operatorname{wt}(b),\a_i^\vee\rangle=0$, then $\t f_i(b)=\t f_i^*(b)$.
\item If $\e_i(b)+\e_i^*(b)+ \langle \operatorname{wt}(b),\a_i^\vee\rangle\geq 1$, then $\e_i^*(\t f_i(b))=\e_i^*(b)$ and $\e_i(\t f_i^*(b))=\e_i(b)$.
\item If $\e_i(b)+\e_i^*(b)+ \langle \operatorname{wt}(b),\a_i^\vee\rangle\geq 2$, then $\t f_i \t f_i^*(b) = \t f_i^* \t f_i (b)$.
 \end{enumerate}
Then $(B,\tilde e_i,\tilde f_i)\cong(B,\tilde e_i^*,\tilde f_i^*)\cong B(\infty)$. Furthermore, if $*$ is the Kashiwara involution, then $\tilde e_i^*=*\tilde e_i *$ under these identifications.
\end{theorem}

Let $t$ be a natural number. We define a highest weight $\mathfrak{sl_2}$ bicrystal $B_t$. Let
\[
 B_t=\{(a,b,c)\in\N^3\mid a+b+c=t\}\cup \{(a,b,0)\in\N^3\mid t-|a-b|\in 2\N\}
\]

The operators $\t f$ and $\t f^*$ act on $B_t$ by the formulae
\[
 \t f(a,b,c)=\begin{cases}
                 (a+1,b,c-1)\quad\text{if $c>0$} \\
                 (a+1,b+1,0)\quad\text{if $c=0$}
                \end{cases} \]
    \[            \t f^*(a,b,c)=\begin{cases}
                                  (a,b+1,c-1)\quad\text{if $c>0$} \\
                                  (a+1,b+1,0)\quad\text{if $c=0$}
                                 \end{cases}
\]

In this crystal, $\epsilon(a,b,c)=a$, $\e^*(a,b,c)=b$ and the weight is given by $$\e(a,b,c)+\e^*(a,b,c)+\langle \operatorname{wt}(a,b,c),\a_i^\vee\rangle=c.$$

The following lemma is straightforward:

\begin{lemma}\label{bt}
 Suppose $(B,\t e_j,\t f_i,\t e_i^*,\t f_i^*)$ satisfies the conditions of Theorem \ref{criterion}.
 Fix $j\in J$. Then any subset of $B$ generated by a single $b\in B$ and the operators $\t e_j$, $\t f_j$, $\t e_j^*$ and $\t f_j^*$ is isomorphic to $B_t$ for some natural number $t$.
\end{lemma}

\begin{proof}[Proof of Theorem \ref{criterion}]
The criterion of
 \cite[Proposition 3.2.3]{kashiwarasaito} is in terms of a strict embedding of crystals $\Phi_i\map{B}{B\otimes B_i}$. It is straightforward from Lemma \ref{bt} to compute that such a strict embedding exists and satisfies the necessary properties to allow us to apply the Kashiwara-Saito criterion and complete the proof.
\end{proof}

 The following lemma is the folded version of \cite[Proposition 7.1(ii)]{laudavazirani} and has the same proof.
 \begin{lemma}\label{7.2ii}
  Suppose that $M\in \sB$ is such that $\e_i^*(\t  f_i M)=\e_i^*(M)=c$. Then $(\t e_i)^c \t f_i M = \t f_i (\t e_i)^c M$.
 \end{lemma}

\begin{proof}
 Let $N= (\t e_i)^c M$. Then $\t f_i M$ is a quotient of $L(i)\circ N \circ L(i^c)$. Applying the exact functor $\Res_{\nu+i,ci}$ yields a surjection
 \[
  \Res_{\nu+i,ci}L(i)\circ N \circ L(i^c) \twoheadrightarrow \Res_{\nu+i,ci} \t f_i M.
 \]
Since $\e_i^*(\t f_i M)=c$, the object $\Res_{\nu+i,ci} \t f_i M$ is nonzero and every simple subquotient is of the form $(\t e_i)^c \t f_i M\otimes L(i^c)$. Therefore there is a surjection
\[
  \Res_{\nu+i,ci}L(i)\circ N \circ L(i^c) \twoheadrightarrow (\t e_i)^c \t f_i M\otimes L(i^c).
 \]
 
Note that $\e_i^*(N)=0$. Therefore in the Mackey filtration of $\Res_{\nu+i,ci}L(i)\circ N \circ L(i^c)$, the only possible subquotients which are not traceless are $L(i)\circ N \otimes L(i^c)$ and $N\circ L(i)\otimes L(i^c)$.
By the induction-restriction adjunction and the fact that $\e_i^*((\t e_i)^c \t f_i M)=0$, there are no homomorphisms from $N\circ L(i)$ to $(\t e_i)^c \t f_i M$. Therefore there must be a surjection
\[
 L(i)\circ N \twoheadrightarrow (\t e_i)^c \t f_i M.
\]
Since $\t f_i N$ is the unique irreducible quotient of $L(i)\circ N$, it must be that we have our desired isomorphism $\t f_i N \cong (\t e_i)^c \t f_i M$.
\end{proof}

\begin{proposition}\label{prop:heart}
 Conditions (1)-(6) in Theorem \ref{criterion} are satisfied by the crystal $\sB$.
\end{proposition}

\begin{proof}
 That condition (1) is satisfied is obvious.
 
 For (2), note that 
$\t f_i \t f_j^*(M)$ and $\t f_j^* \t f_i (M)$ are both simple quotients of $L_i\circ M\circ L_j$. It suffices to prove that the underlying $R$-module of $L_i\circ M\circ L_j$ has a simple head, which then forces these two quotients to be identical. This follows from the classification of irreducible modules in terms of semicuspidal decompositions in \cite[\S 2]{tingleywebster}, where we consider a convex order with all simple roots in $i$ at one extreme and all simple roots in $j$ at the other extreme.

The condition (3) does not involve the diagram automorphism, so follows from the corresponding result in the unfolded case, namely \cite[Proposition 6.7(v)]{laudavazirani}.

We now turn our attention to (4). Suppose $M$ is such that  $\e_i(M)+\e_i^*(M)+ \langle \operatorname{wt}(M),\a_i^\vee\rangle=0$. We apply the inequality (3) to $\t f_i(M)$ as well as the obvious inequalities $\e_i(\t f_i(M))-\e_i(M)\leq 1$ and $\e^*_i(\t f_i(M))-\e_i^*(M)\leq 1$ to deduce that $\e_i^* (\t f_i(M))= \e_i (M)+1$.

Now consider $M\circ L_i$. The object $\t f^*(M)$ is both the head and the unique subquotient with maximal $\epsilon_i^*$. Taking duals shows that the socle of $L_i\circ M$ is the unique subquotient with maximal $\epsilon_i^*$. But the identity $\e_i^* (\t f_i(M))= \e_i (M)+1$ shows that the head, $\t f_i(M)$, is the unique subquotient with maximal $\epsilon_i^*$. Therefore $L_i\circ M$ is irreducible, so the $R$-matrix induces an isomorphism (in $\C$) between $M\circ L_i$ and $L_i\circ M$. This is an isomorphism between $\t f_i^*(M)$ and $\t f_i(M)$, as desired.

The condition (5) also does not involve the diagram automorphism, so also follows from the corresponding result in the unfolded case in \cite{laudavazirani}.

For condition (6), first apply condition (5) to deduce $\e_i^*(\t f_i M)=\e_i^*(M)$. This implies that $\t f_i M$ also satisfies the condition (5). Now applying (5) to both $M$ and $\t f_i M$ yields $\e_i (\t f_i^* M)=\e_i M$ and  $\e_i (\t f_i^* \t f_i M)=\e_i (\t f_iM)$. With these two equations, we apply Lemma \ref{7.2ii} twice to reach the desired conclusion.
\end{proof}

\begin{lemma}\label{efcommute}
 Suppose $i,j\in J$ and $b\in \sB$. Then $\tilde e_i \tilde f_j^* b$ is either equal to $\tilde f_j^* \tilde e_i b$ or $b$.
\end{lemma}

\begin{proof}
 If $i=j$ then this is an immediate consequence of Lemma \ref{bt} and Proposition \ref{prop:heart} since this result holds in the crystal $B_t$. So now suppose that $i\neq j$. Then $\epsilon_i(b)=\epsilon_i(\tilde f_j^*b)$ so either $\tilde e_i \tilde f_j^* b$ and $\tilde f_j^* \tilde e_i b$ are both zero, or both nonzero. In the latter case, applying condition (3) to $\tilde e_i \tilde f_j^* b$ implies the result.
\end{proof}

Let $[1]$ be the object $(k,\id)$ in $\C_0$. It lies in $\sB$. If $n$ is odd this is the only element of $\sB$ of weight zero. If $n$ is even, there is exactly one other weight zero element of $\sB$, which we denote $[-1]$.

\begin{lemma}\label{bl=br}
 Let $\sB_L$ be the subset of $\sB$ consisting of all elements of the form $\tilde f_{i_1}\cdots \tilde f_{i_k} [1]$. Let $\sB_R$ be the subset of $\sB$ consisting of all elements of the form $\tilde f^*_{i_1}\cdots \tilde f^*_{i_k} [1]$. Then $\sB_L = \sB_R$.
\end{lemma}

\begin{proof}
 What we have to prove is that it is impossible to have an equation of the form.
 \[
  \tilde f_{i_1}\cdots \tilde f_{i_k} [1] = \tilde f^*_{j_1}\cdots \tilde f^*_{j_k} [-1]
 \]
 Suppose for want of a contradiction that such an equation holds. We apply $\tilde e_{i_1}$ to this equation and repeatedly use Lemma \ref{efcommute} on the right hand side to either arrive at a smaller counterexample or a direct contradiction.
\end{proof}

\begin{theorem}\label{crystal}
 The bicrystal $\sB_L=\sB_R$ is isomorphic to $B(\infty)$.
\end{theorem}

\begin{proof}
 It is clear that $\sB_L=\sB_R$ is highest weight. It is a bicrystal by Lemma \ref{bl=br}. Proposition \ref{prop:heart} shows that it satisfies the conditions (1)-(6) in the statement of Theorem \ref{criterion}. Applying Theorem \ref{criterion} completes the proof.
\end{proof}

We will now write $\sB_+$ for the bicrystal $\sB_L=\sB_R$.

\begin{remark}
 Our use of \cite[Proposition 6.7(v)]{laudavazirani} is the only place in this paper where we use the necessary fact that the leading coefficients of $Q_{ij}(u,v)$ are nonzero.
\end{remark}

\begin{corollary}
 The Grothendieck group $K(\P_\nu)$ has dimension
 \[
  \sum_{\nu\in \N J}\dim K(\P_\nu) t^\nu = \prod_{\a\in\Phi^+} (1-t^{\a})^{-\operatorname{mult}(\a)}.
 \]
\end{corollary}

\section{Proof of the Main Theorem}\label{proof}

We have natural bases $\B$ and $\B^*$ of $K(\P)$ and $K(\L)$.
Specifically $\B^*$ consists of the classes of objects of $\sB_+$. This is a basis by Theorem \ref{classifysimples}, together with the fact that every irreducible $R(\nu)$-module is isomorphic to its dual up to a grading shift. The basis $\B$ consists of the classes of projective covers of elements of $\sB_+$.

\begin{theorem}\label{dual2}
 The bases $\B$ and $\B^*$ are dual to each other.
\end{theorem}

\begin{proof}
 Since $k$ is algebraically closed, Schur's Lemma implies that the endomorphism algebra of any simple object is $k$. Therefore $\Hom(P,L)$ is either 0 or $k$ when $P$ is indecomposable projective and $L$ is simple. As the pairing between $K(\P)$ and $K(\L)$ is realised by the Hom pairing (\ref{pairing}), this implies that $\B$ and $\B^*$ are dual to each other.
\end{proof}

It is natural to ask the following question

\begin{question}\label{integrality}
Do the structure constants for multiplication in $K(\P)$ with respect to the basis $\B$ lie in $\Z[q,q\inv]$?
\end{question}

We remark that existing proofs of similar facts in the literature \cite{elias,lusztigbook} require as input deep geometric results that we do not have access to. They imply that if $\operatorname{char}{k}=0$ and $R(\nu)$ is of geometric type (which means that $Q_{ij}(u,v)=\pm(u-v)^d$ for all $i,j$) then the answer to Question \ref{integrality} is yes. We are only able to prove the following weaker version, which answers Question \ref{integrality} in the affirmative in most interesting cases.

\begin{lemma}\label{mstructureconstants}
The structure constants for multiplication in $K(\P)$ with respect to the basis $\B$ are in $\Z[\zeta_n+\zeta_n\inv,q,q\inv]$.
\end{lemma}

\begin{proof}
 
%
%
 Write $[P][Q]=\sum_{R\in \B} a_{PQ}^R [R]$ where a priori the $a_{PQ}^R$ lie in $\Z[\zeta_n,q,q\inv]$. On the Grothendieck group, the duality $\D$ sends $\zeta_n$ to $\zeta_n\inv$, which we will denote by a bar involution on $\Z[\zeta_n,q,q\inv]$ fixing $q$. 
 
 Since $\D (P\circ Q)\cong (\D P)\circ (\D Q)$ by Lemma \ref{lem:dualproj}, we have
 \[
  \sum_{R\in \B} a_{PQ}^R [R]=[P][Q]=[\D P][\D Q]=[\D(P\circ Q)]=\sum_R \overline{a_{PQ}^R} [\D R] = \overline{a_{PQ}^R} [R].
 \]
This implies that each coefficient $a_{PQ}^R$ is invariant under the bar involution so lies in $\Z[\zeta_n+\zeta_n\inv,q,q\inv]$ as desired.
\end{proof}

Let $\j=(j_1,j_2,\ldots)$ be a sequence of elements of $J$ in which each element occurs infinitely often. Let $\N^\infty$ be the set of sequences $\{c_m\}_{m=1}^\infty$ of natural numbers such that $c_m=0$ for all but finitely many $m$.

Suppose $b\in B(\infty)$. Define a sequence $\{c_m\}_{m=1}^\infty$ of natural numbers by 
\[
 c_m=\epsilon_{j_m}(\tilde e_{j_{m-1}}^{c_{{m-1}}}(\cdots\tilde e_{j_1}^{c_{1}}(b)\cdots )).
\]

This defines a function $\iota\map{B(\infty)}{\N^\infty}$ which is known to be injective by \cite[Theorem 2.2.1]{kashiwara}. Since $\sB_+\cong B(\infty)$, we can view $\iota$ as an inclusion from $\sB_+$ to $\N^\infty$.

If $c\in\N^\infty$, we define $P^{(c)}=P_{j_1}^{(c_1)}\circ P_{j_2}^{(c_2)} \circ \cdots$.

\begin{lemma}\label{triangularity}
 For $c\in \N^\infty$ and $L\in\sB_+$, $\Hom(P^{(c)},L)=0$ unless $\iota(L)\geq c$ in lexicographical order. If $\iota(L)=c$ then $\Hom(P^{(c)},L)\cong k$.
\end{lemma}

\begin{proof}
 Write $P^{(c)}=P_{j_1}^{(c_1)}\circ Q$. By adjunction,
 \[
  \Hom(P^{(c)},L)\cong \Hom(P_{j_1}^{(c_1)}\otimes Q,\Res_{c_1j_1,\nu-c_1j_1} L).
 \]
If $\e_{j_1}(L)<c_1$, this is zero. If $\e_{j_1}(L)=c_1$, then by Lemma \ref{3.8}, $\Res_{c_1j_1,\nu-c_1j_1} L\cong L(j_1^{c_1})\otimes \t 
e_{j_1}^{c_1} L$. Therefore
\[
 \Hom(P^{(c)},L)\cong \Hom(Q,\t e_{j_1}^{c_1} L).
\]
So unless $\iota(L)> c$ in lexicographical order, we have either computed $\Hom(P^{(c)},L)$ or are in a position where we can proceed by induction.
\end{proof}

\begin{lemma}\label{cstructureconstants}
 The structure constants for comultiplication in $K(\P)$ with respect to the basis $\B$ are in $\Z[q,q\inv]$.
\end{lemma}

\begin{proof}
By Lemma \ref{triangularity}, the set $\{[P^{(\iota(b))}]\}$ is a basis of $K(\P)$. Since the objects $P^{(\iota(b))}$ are all self-dual, these also form a basis of $\kk$. We will now compute $r([P^{(c)}])$ to prove the result.

From Theorem \ref{7.3}, $[c]!_q[P^{(c)}]=[P^c]$ so it suffices to compute $r([P^c])$. Theorem \ref{mackey} tells us that modulo traceless subquotients (which are zero in the Grothendieck group), all subquotients of $\Res P^c$ are a grading shift of $P^{c'}\otimes P^{c''}$. Thus all structure constants lie in $\Q(q)$. This is enough to complete the proof since $\Q(q)\cap \Z[\zeta_n,q,q\inv]=\Z[q,q\inv]$.
\end{proof}

\begin{proposition}
 The twisted bialgebra structures on $\oplus_{\nu\in \N J}K(\P_\nu)$ and $\oplus_{\nu\in\N J}K(\L_\nu)$ restrict to twisted bialgebra structures on $\kk$ and $\kk^*$ respectively. The corresponding twisted bialgebras $\kk$ and $\kk^*$ are graded duals of each other.
\end{proposition}

\begin{proof}
 That $\kk$ is a bialgebra is Lemmas \ref{mstructureconstants} and \ref{cstructureconstants}. Since $k$ is algebraically closed, the endomorphism algebra of any simple object is $k$. Therefore the Hom pairing is a perfect pairing between $\kk$ and $\kk^*$. Hence the statement for $\kk^*$ follows from that for $\kk$.
\end{proof}

Let $'\f_{\Q(q)}$ be the free $\Q(q)$-algebra generated by $\th_j$. Let $'\f$ be the $\Z[q,q\inv]$-subalgebra generated by the divided powers $\th_j^{(n)}$. There is a canonical surjection from $'\f$ to $\f$.

Define $\chi\map{'\f\otimes A}{\kk}$ to be the unique algebra homomorphism such that $\chi(\theta_j)=[P_j]$ for all $j\in J$.

\begin{lemma}
 $\chi$ is a homomorphism of coalgebras.
\end{lemma}

\begin{proof}
To show that $\chi$ is a coalgebra homomorphism, we have to show that the coproduct on $\kk$ satisfies the properties
\begin{enumerate}
 \item $r([P_j])=[P_j]\otimes 1 + 1 \otimes [P_j]$
\item $r(xy)=r(x)r(y)$ where the multiplication on the right hand side is the twisted one given in (\ref{twistedproduct}).
\end{enumerate}
The first of these facts is straightforward, while the second follows from Theorem \ref{mackey} (the Mackey filtration).
\end{proof}

Recall the ring $A$ of coefficients introduced in $\S 6$.

\begin{lemma}\label{surjective}
 The homomorphism $\chi\otimes A\map{'\f\otimes A}{\kk}$ is surjective.
\end{lemma}

\begin{proof}
Since $[P_j^{(c)}]=\chi(\th_j^{(c)})$, the elements $\{P^{(\iota(b))}\}$ for $b\in B(\infty)$ all lie in the image of $\chi$.
 Lemma \ref{triangularity} shows that if we expand $\{P^{(\iota(b))}\}$ in the basis $\B$ of classes of indecomposable projectives, the coefficients that appear form a unitriangular matrix. In particular this matrix is invertible over $A[q,q\inv]$ which implies the surjectivity of $\chi\otimes A$.
\end{proof}

\begin{lemma}\label{isometry}
 $\chi$ is an isometry.
\end{lemma}

\begin{proof}
 To show that $\chi$ is an isometry it suffices to show that the pairing $(\cdot,\cdot)$ on $\kk$ defined by (\ref{projpairing}) satisfies the properties
 \begin{enumerate}
  \item $([P_j],[P_j])=(1-q^{j\cdot j})\inv$,
  \item $(x,yz)=(r(x),y\otimes z)$,
  \item $(xy,z)=(x\otimes y,r(z))$.
 \end{enumerate}
The first of these conditions is Lemma \ref{7.1}. The third follows from the Induction-Restriction adjunction. The second follows from the third since the form $(\cdot,\cdot)$ is symmetric.
\end{proof}


\begin{proof}[Proof of Theorem \ref{projiso}]
 Recall that $\f$ is the quotient of $'\f$ by the radical of $\langle\cdot,\cdot\rangle$. Let $R$ be the radical of $(\cdot,\cdot)$ in $\kk$. Since $\chi$ is an isometry, it induces an injective homomorphism $\ga\map{\f}{\kk/R}$. 
 Theorem \ref{crystal} and \cite[Theorem 33.1.3]{lusztigbook} compute the graded ranks of the free modules $\kk$ and $\f$ respectively. These ranks are the same, so it must be that $R$ is zero. and $\gamma$ is an injective homomorphism from $\f$ to $\kk$. Surjectivity follows from Lemma \ref{surjective} and the rest is straightforward.
\end{proof}

Theorem \ref{simiso} follows immediately from Theorem \ref{projiso}.


\section{Bases of Canonical Type}

Let $\sigma$ be the antiautomorphism of $\f$ that fixes the Chevalley generators $\th_j$. This is induced from an automorphism $\sigma$ of each $R(\nu)$ which sends $e_{\ii_1,\ldots,\ii_n}$ to $e_{\ii_n,\ldots,\ii_1}$, $y_j$ to $y_{n+1-j}$ and $\tau_k$ to $-\tau_{n-k}$.
 The bar involution of $\f$ is the automorphism fixing the generators $\th_j$ and sending $q$ to $q\inv$. Both $\sigma$ and the bar involution induce corresponding involutions on $\f^*$, which we will also call $\sigma$ or the bar involution respectively.
 
 The following definition is lifted from \cite{baumann} where it appears in the special case where $q$ is specialised to 1.
 
 \begin{definition}\label{def:can}
A basis $\mathbf B$ of $\mathbf f$ is said to be of canonical type
if it satisfies the conditions~(\ref{it:BOCTa})--(\ref{it:BOCTf}) below:
\begin{enumerate}
\item
\label{it:BOCTa}
The elements of $\mathbf B$ are weight vectors.
\item
\label{it:BOCTb}
$1\in\mathbf B$.
\item
\label{it:BOCTc}
Each right ideal $(\theta_j^p\mathbf f\otimes \Q(q))\cap \f$ is spanned by a subset of
$\mathbf B$.
\item
\label{it:BOCTd}
In the bases induced by $\mathbf B$, the left multiplication by
$\theta_j^{(p)}$ from $\mathbf f/\theta_j\mathbf f$ onto
$\theta_j^{(p)}\mathbf f/\theta_j^{(p+1)}\mathbf f$ is given by a
permutation matrix.
\item
\label{it:BOCTe}
$\mathbf B$ is stable by $\sigma$.
\item
\label{it:BOCTf}
$\mathbf B$ is stable under the bar involution.
\end{enumerate}
\end{definition}

There is a dual notion. Here we let $r_{j^p}$ be the adjoint of left multiplication by $\th_j^{(p)}$.

\begin{definition}\label{def:dualcan}
A basis $\B^*$ of $\f^*$ is said to be of dual canonical type if it satisfies the conditions ~(\ref{it:BOdCTa})--(\ref{it:BOdCTf}) below:
\begin{enumerate}
\item
\label{it:BOdCTa}
The elements of $\mathbf B^*$ are weight vectors.
\item
\label{it:BOdCTb}
$1\in\mathbf B^*$.
\item
\label{it:BOdCTc}
Each $\ker(r_{j^p})$ is spanned by a subset of $\B^*$.
\item
\label{it:BOdCTd}
In the bases induced by $\mathbf B^*$, the map
\[
 r_{j^p}\map{\ker(r_{j^{p+1}})/\ker(r_{j^{p}})}{\ker(r_{j})}
\]
is given by a
permutation matrix.
\item
\label{it:BOdCTe}
$\mathbf B^*$ is stable by $\sigma$.
\item \label{it:BOdCTf}
$\mathbf B^*$ is stable under the bar involution.
\end{enumerate}
\end{definition}

The algebra $\f$ acts on $\f^*$ where $\th_j$ acts by the endomorphism $r_j$. If $\B^*$ is a basis of dual canonical type, then the specialisation of $\B^*$ at $q=1$ is a perfect basis of the unipotent coordinate ring $\Z[N]$ (which is the specialisation of $\f^*$ at $q=1$) in the sense of \cite[Definition 5.30]{berensteinkazhdan}. The other results of \cite[\S 5]{berensteinkazhdan} then prove that the crystal obtained from any basis of dual canonical type is isomorphic to $B(\infty)$.

\begin{theorem}\label{dual1}
 A basis $\B$ of $U_q(\n)$ is of canonical type if and only if its dual basis $\B^*$ of $\A_q(\n)$ is of dual canonical type.
\end{theorem}

\begin{proof}
Let $A$ be the matrix representing the linear transformation from $\f_\nu$ to $\f_{\nu+pj}$ that is left-multiplication by $\th_j^{(p)}$ with respect to the basis $\B$. Condition (3) in Definition \ref{def:can} is equivalent to $A$ being in block form $\left(\begin{smallmatrix} 0 \\ X \end{smallmatrix}\right)$ where the subset of $\B$ spanning the right ideal $\theta_i^p\mathbf f\otimes \Q(q)\cap \f$ corresponds to the rows in $X$. Furthermore $X$ is invertible over $\Q(q)$ since this linear transformation is injective.

Let $B$ be the matrix representing the linear transformation $r_{j^p}$ from $\f_{\nu+pj}^*$ to $\f_\nu^*$ with respect to the basis $\B^*$. Condition (3) in Definition \ref{def:dualcan} is equivalent to $B$ being in block form $\left(\begin{smallmatrix} 0 & Y \end{smallmatrix}\right)$ with $Y$ invertible over $\Q(q)$ where the subset of $\B^*$ spanning $\ker(r_{j^p})$ corresponds to the columns not part of $Y$.

Since passing between the matrices representing adjoint operators corresponds to taking the transpose of a matrix, the two conditions labelled (3) in Definitions \ref{def:can} and \ref{def:dualcan} are equivalent. The equivalence of the other conditions is obvious.
\end{proof}

\begin{theorem}\cite[\S 2.3]{baumann}\label{basisofvla}
 Let $\la\in P^+$. Let $V(\la)$ be a highest weight module of the quantum group $U_q(\g)$ constructed from $(J,\cdot)$ with highest weight $\la$ and let $v_\la$ be a highest weight vector. Let $\B$ be a basis of canonical type. Then the set
 \[
  \{bv_\la \mid b\in \B,bv_\la\neq 0\}
 \]
 is a basis of $V(\la)$.
\end{theorem}

\begin{proof}
 By \cite[Corollary 6.2.3(a)]{lusztigbook}, the map $x\mapsto x^-v_\la$ is a surjective map from $\f$ to $V(\la)$ with kernel the left ideal
  \[
   K=\sum_{j\in J} \f \th_j^{\langle \a_j^\vee,\la\rangle +1}.
  \]
By conditions (\ref{it:BOCTc}) and (\ref{it:BOCTe}) in the definition of a basis of canonical type, $K$ is spanned by elements of $\B$. Hence $\B$ induces a basis of the quotient $V(\la)$.
\end{proof}

%


\begin{theorem}\label{important}
 Let $M\in \sB_+$ and $p\leq \epsilon_j(M)$ be a natural number. Then
 \begin{equation}\label{rjofm}
 r_{j^p}([M])={\epsilon_j(M) \brack p}_j[\tilde e_j^p M]+ \sum_{\substack{N\in \sB_+ \\ \epsilon_j(N)<\epsilon_j(M)-p}}a_N [N]
\end{equation}
 for some $a_N\in \Z[q,q\inv]$.
\end{theorem}

\begin{proof}
 Let $m=\e_j(M)$. By Lemma \ref{3.8}, $ \Res_{mj,\nu-mj}(M)\cong  L(j^m)\otimes X$ for some irreducible object $X$ of $\C_{\nu-mj}$.
 
 Now suppose that $L(i^p)\otimes Y$ is a simple composition factor of $\Res_{pm,\nu-pm}M$. Note that $\e_j(Y)\leq m-p$. If $\e_j(Y)<m-p$ then this term contributes in an allowable way to the sum in (\ref{rjofm}). If $\e_j(Y)=m-p$ then again by Lemma \ref{3.8}, $\Res_{(m-p)j,\nu-(m-p)j} Y\cong L(j^{m-p})\otimes Z$ for some irreducible object $Z$. Since the restriction functor is exact we must have $X\cong Z$.
 
 As $X\cong \t e_j^m(M)$, this identifies $Y$ as $\t f_j^{m-p}\t e_j^m(M)=\t e_j^p(M)$.
 
 It remains to compute the coefficient of $[\t e_j^pM]$ in the expansion (\ref{rjofm}), which follows from (\ref{onecoloureqn}).
\end{proof}

%
%
%
%

\begin{theorem}\label{part3}
 The set $\B^*\cap \ker(r_{i^p})$ is a basis of $\ker(r_{i^p})$.
\end{theorem}

\begin{proof}
 It suffices to show that $r_{i^p}$ is injective on
 \[
\operatorname{span}   \langle [M] \mid M\in \sB_+ \text{ and } \e_i(M)\geq p\rangle.
 \]
 
 Since the crystal operator $\tilde e_i^p$ is injective on the set of $b\in \sB_+\cong B(\infty)$ satisfying $\epsilon_i(b)\geq p$, this injectivity follows from Theorem \ref{important} above.
 \end{proof}
%
%
%
%

\begin{theorem}
 The basis $\B^*$ is a basis of dual canonical type.
\end{theorem}

\begin{proof}
 Conditions (1) and (2) in Definition \ref{def:dualcan} are obvious. Condition (3) is Theorem \ref{part3}. Condition (4) follows from Theorem \ref{important}, noting that if $[M]\in \ker(r_{j^{p+1}})\setminus \ker(r_j^{p})$, then $\e_j(M)=p$. The stability under $\sigma$ is because $\sigma$ is induced by the automorphism $\sigma$ of $R(\nu)$. The stability under the bar involution is because $\B^*$ consists of the classes of self-dual objects.
\end{proof}

 \begin{corollary}\label{canonicaltype}
  The KLR basis $\B$ is a basis of canonical type.
 \end{corollary}
 \begin{proof}
  This is a consequence of Theorems \ref{canonicaltype}, \ref{dual1} and \ref{dual2}.
 \end{proof}

 From Theorem \ref{basisofvla}, we obtain in addition the following Corollary.

\begin{corollary}\label{basis}
 The KLR basis $\B$ induces a basis in any irreducible highest weight module.
\end{corollary}

\begin{remark}
When there is no diagram automorphism, a categorified version of Corollary \ref{basis} is proved in \cite{kangkashiwara} and \cite{webster}.
\end{remark}

When $C$ is a symmetric Cartan matrix, the field $k$ is of characteristic zero and the polynomials $Q_{ij}(u,v)$ are all of the form $\pm (u-v)^{-c_{ij}}$, the basis $\B$ is the same as Lusztig's canonical basis, or equivalently Kashiwara's lower global basis. This follows from comparing the geometric interpretation of the KLR algebras as Ext algebras due to \cite{vv,rouquier2} and Lusztig's construction of the canonical basis in terms of perverse sheaves in \cite{lusztigbook}.
A relevant discussion for how to geometrically construct projective modules in this situation can be found in \cite[\S 8.7]{chrissginzburg}.
%
%
%
%
%
%
%

%

\begin{thebibliography}{KKKO15}

\bibitem[B]{baumann}
Pierre Baumann.
\newblock The canonical basis and the quantum {F}robenius morphism.
\newblock {\em \arXiv{1201.0303}}.

\bibitem[BK]{berensteinkazhdan}
Arkady Berenstein and David Kazhdan.
\newblock Geometric and unipotent crystals. {II}. {F}rom unipotent bicrystals
  to crystal bases.
\newblock In {\em Quantum groups}, volume 433 of {\em Contemp. Math.}, pages
  13--88. Amer. Math. Soc., Providence, RI, 2007. \arxiv{math/0601391}

\bibitem[CG]{chrissginzburg}
Neil Chriss and Victor Ginzburg.
\newblock {\em Representation theory and complex geometry}.
\newblock Birkh\"auser Boston, Inc., Boston, MA, 1997.

\bibitem[E]{elias}
Ben Elias,
\newblock Folding with Soergel bimodules.
\newblock In {\em {Categorification and higher representation theory}}, pages 287--332. Amer. Math. Soc., Providence, RI, 2017.
\arxiv{1602.08449}

\bibitem[G]{grojnowski}
Ian Grojnowski
\newblock Affine $\mathfrak{sl}_p$ controls the representation theory of the symmetric group and related Hecke algebras.
\arxiv{math/9907129}

\bibitem[Ka]{kashiwara}
\newblock The crystal base and Littelmann's refined Demazure character formula.
\newblock {\em Duke Math. J.} 71(3), 839--858, 1993.

\bibitem[KK]{kangkashiwara}
Seok-Jin Kang and Masaki Kashiwara.
\newblock Categorification of highest weight modules via
  {K}hovanov-{L}auda-{R}ouquier algebras.
\newblock {\em Invent. Math.}, 190(3):699--742, 2012.
 \arXiv{1102.4677}

\bibitem[KKK]{kkk}
  Seok-Jin Kang, Masaki Kashiwara, and Myungho Kim.
\newblock Symmetric quiver {H}ecke algebras and {R}-matrices of quantum affine
  algebras.
\newblock {\em Invent. Math.}, 211(2):591--685, 2018.
{\em \arxiv{1304.0323}}.


\bibitem[KKKO1]{kkko}
Seok-Jin Kang, Masaki Kashiwara, Myungho Kim, and Se-jin Oh.
\newblock Simplicity of heads and socles of tensor products.
\newblock {\em Compos. Math.}, 151(2):377--396, 2015. \arxiv{1404.4125}


\bibitem[KKKO2]{kkko2}
Seok-Jin Kang, Masaki Kashiwara, Myungho Kim, and Se-jin Oh.
\newblock Monoidal categorification of cluster algebras.
\newblock {\em J. Amer. Math. Soc.}, 31(2):349--426, 2018.
\newblock {\em \arxiv{1801.05145}}.

\bibitem[KL1]{khovanovlauda}
Mikhail Khovanov and Aaron~D. Lauda.
\newblock A diagrammatic approach to categorification of quantum groups. {I}.
\newblock {\em Represent. Theory}, 13:309--347, 2009.
\arXiv{0803.4121}

\bibitem[KL2]{kl2}
Mikhail Khovanov and Aaron~D. Lauda.
\newblock A diagrammatic approach to categorification of quantum groups {II}.
\newblock {\em Trans. Amer. Math. Soc.}, 363(5):2685--2700, 2011.
\arXiv{0804.2080}

\bibitem[Kl]{kleshchevbook}
Alexander Kleshchev.
\newblock {\em Linear and projective representations of symmetric groups},
  volume 163 of {\em Cambridge Tracts in Mathematics}.
\newblock Cambridge University Press, Cambridge, 2005.

\bibitem[KS]{kashiwarasaito}
Masaki Kashiwara and Yoshihisa Saito.
\newblock Geometric construction of crystal bases.
\newblock {\em Duke Math. J.}, 89(1):9--36, 1997.
\arXiv{q-alg/9606009}

\bibitem[L]{lusztigbook}
George Lusztig.
\newblock {\em Introduction to quantum groups}, volume 110 of {\em Progress in
  Mathematics}.
\newblock Birkh\"auser Boston Inc., Boston, MA, 1993.

\bibitem[LV]{laudavazirani}
Aaron~D. Lauda and Monica Vazirani.
\newblock Crystals from categorified quantum groups.
\newblock {\em Adv. Math.}, 228(2):803--861, 2011.
\arXiv{0909.1810}

\bibitem[M]{maksimau}
Ruslan Maksimau.
\newblock Canonical basis, {KLR} algebras and parity sheaves.
\newblock {\em J. Algebra}, 422:563--610, 2015.
\arXiv{1301.6261}

\bibitem[Mc]{cluster}
Peter~J. McNamara.
\newblock Cluster monomials are dual canonical.
\newblock {\em Preprint}.

\bibitem[R1]{rouquier}
Rapha\"el~Rouquier.
\newblock 2-{K}ac-{M}oody algebras.
\newblock {\em \arXiv{0812.5023}}.

\bibitem[R2]{rouquier2}
Rapha\"el Rouquier.
\newblock {Quiver Hecke algebras and 2-Lie algebras.}
\newblock {\em Algebra Colloq.}, 19(2):359--410, 2012.
 \arXiv{1112.3619}

\bibitem[TW]{tingleywebster}
Peter Tingley and Ben Webster.
\newblock Mirkovi\'{c}-{V}ilonen polytopes and {K}hovanov-{L}auda-{R}ouquier
  algebras.
\newblock {\em Compos. Math.}, 152(8):1648--1696, 2016.
  \arxiv{1210.6921}.

\bibitem[VV]{vv}
Michela~Varagnolo and Eric~Vasserot.
\newblock Canonical bases and {KLR}-algebras.
\newblock {\em J. Reine Angew. Math.}, 659:67--100, 2011.
\arXiv{0901.3992}

\bibitem[W]{webster}
Ben Webster.
\newblock Canonical bases and higher representation theory.
\newblock {\em Compos. Math.}, 151(1):121--166, 2015.
\arXiv{1209.0051}

\end{thebibliography}
\def\cprime{$'$}

\end{document}